\documentclass[12pt]{article}
\usepackage{amssymb}
\usepackage{amsmath}
\usepackage{tikz}
\usepackage{longtable}
\usepackage{latexsym}
\usepackage{longtable}

\newcommand{\bea}{\begin{eqnarray}}
\newcommand{\eea}{\end{eqnarray}}
\newcommand{\be}{\begin{equation}}
\newcommand{\ee}{\end{equation}}

\newtheorem{proposition}{Proposition}[section]

\newenvironment{proof}[1][Proof]{\textbf{#1.} }{\rule{0.5em}{0.5em}}
\makeatletter\def\theequation{\arabic{section}.\arabic{equation}}\makeatother
\usepackage{pst-node}
\usepackage{pstricks}
\usepackage{graphicx}
\usepackage{latexsym}

\title{\bf Matrix Invariants of Finite Metric Spaces}
\author{Ayse Humeyra Bilge (ayse.bilge@khas.edu.tr),\\ Metehan Incegul (metehanincegul@gmail.com)}
\date{}
\begin{document}
	\maketitle


 \begin{abstract}
Finite metric spaces are characterized by a polyhedral cone defined in terms of the positivity of the distance functions and the triangle inequalities. Their classification is based on the decomposition of an associated polyhedral cone, called the ``metric fan".  The complete classification of $n$-point  metric spaces is available only for $n\le 6$.  As the number of classes increases rapidly with the number of elements, it is desirable to have coarser equivalence class decompositions based on certain invariants of finite metric spaces.   If $(X,d)$ is  a finite metric space with elements $P_i$ and  with distance functions $d_{ij}$, the Gromov product at $P_i$ is defined as $\Delta_{ijk}=1/2(d_{ij}+d_{ik}-d_{jk})$.  Assuming that the set of Gromov product at $P_i$ has a unique smallest element $\Delta_{ijk}$, the association of the edge $P_jP_k$ to $P_i$ defines the ``Gromov product structure".   The ``pendant-free" reduction of the  finite metric space is the graph obtained by removing the edges $P_jP_k$ corresponding to the minimal Gromov products $\Delta_{ijk}$ at $P_i$.
In the present work, we define a matrix representation for a  Gromov product structure $S$ on an
$n$-point metric space, by $n\times n$ matrix  $G_S$.
We prove that if two metric spaces  have Gromov product structures that can be mapped to each other by a permutation of the indices, then their  matrices are similar via the corresponding permutation matrix.
Matrix invariants of $G_S$ are  used to define subclasses of Gromov product structures and their application to $n=5$ and $n=6$-point spaces  are given.
  \end{abstract}


\ Finite metric spaces,\ Gromov product decomposition, Matrix invariants. 

\section{Introduction}

Let $(X,d)$ be a finite metric space with $n$ elements $P_i$, $i=1,\dots,n \ (n \geq 3)$ and let $d_{ij}$ be the distance between $P_i$ and $P_j$.
A metric space with $n$ points is characterized by its set of distance functions $d_{ij}$ subject to the conditions
$d_{ij}=d_{ji}$, $d_{ij}\ne0$ for $i\ne j$ and the triangle inequalities
$d_{ij}+d_{ik}-d_{jk}\ge 0$.
If none of the triangle equalities are saturated,  then, the set
$$\{d_{ij}\in R^{n(n-1)/2}\ | \ d_{ij}>0,\quad d_{ij}+d_{ik}-d_{jk}>0 \}$$
is the interior of a convex polyhedral cone in $R^{{n(n-1)/2}}$.
The decomposition of an associated polyhedral cone called the ``metric fan" is used to classify finite metric spaces \cite{Sturm}.
The number of metrics obtained via this classification  increases rapidly with $n$; there are  $1$, $3$ and $339$ classes for $n=4$, $n=5$ and $n=6$, respectively \cite{Koolen}, \cite{Sturm}.

In previous work we have defined  ``Gromov product structures" as an alternative approach to the study of finite metric spaces \cite{BCK2017}. It was shown that, for $n=4$ and $n=5$,  the classification by the equivalence of Gromov product structures coincides with the classification given by the decomposition of the metric fan, but for $n\ge 6$ it is strictly coarser.  In fact, for $n=6$, $n=7$  and $n=8$ we obtained
$26$, $431$ and $11470$ classes; the Gromov product structures for $6$-point spaces are given in \cite{BCK2017}, the ones for $n=7$ and $n=8$ are yet unpublished.    
In the present work, we define a matrix representation for Gromov product structures and define certain matrix invariants that would allow a coarser classification of finite metric spaces.
We start by giving basic definitions.

The quantity $\Delta_{ijk}$ defined by
$$\Delta_{ijk}={\textstyle \frac{1}{2}}(d_{ij}+d_{ik}-d_{jk})\eqno(1)$$
is called the ``excess" or the Gromov product of the triangle $(P_i,P_j,P_k)$ at the vertex $P_i$ \cite{VJ}. In \cite{BCK2017},  a metric space is called
$\Delta$-generic, if the set of all Gromov products at a fixed vertex $P_i$
has a unique smallest element (for $i=1,\dots,n$).    The ``Gromov Product Structure" is then defined  as the collection of minimal Gromov products at each $P_i$, up to a permutation of the $P_i$'s.
Two $\Delta$-generic finite metric spaces   are said to be  $\Delta$-equivalent, if  the collection of minimal Gromov products at
each $P_i$ is the same, up to a permutation of the $P_i$'s.
An algorithm for determining generic Gromov product structures and finding
their equivalence classes is presented in \cite{IB_algorithm}.

The Gromov product structure on an $n$-point space, leads to a selection of a subgraph of the complete graph $K_n$ as follows. If the Gromov product $\Delta_{ijk}$ is minimal at $P_i$, we may assume that this minimal value is zero, i.e, $d_{jk}=d_{ij}+d_{ik}$. Thus the distance between  $P_j$ and $P_k$ is realized by the path consisting of the edges $P_jP_i$ and $P_iP_k$ and the edge $P_jP_k$ can be removed.  The graph obtained by removing the edges $P_jP_k$ corresponding to the minimal Gromov products $\Delta_{ijk}$ at $P_i$ is called the ``pendant-free reduction" of the complete graph.
The Gromov product structure determines the pendant-free reduction but contains more information: For the   Gromov product types $I_1$ and $I_2$ of Table 2, the set of removed edges, $\{P_1P_3, P_2P_4,  P_5P_6\}$, form a completely disconnected graph, thus their pendant-free reductions are equivalent as graphs, although their Gromov product structures are inequivalent.

In the present work, in Section 2,  we define the matrix representation of a Gromov product structure as an $n\times  n$  matrix and prove that equivalent $\Delta$-structures give rise to matrices that are similar via a permutation matrix.
In Section 3,   we describe the matrix  invariants and illustrate the method for  $6$-point spaces in Section 4.  We present our concluding remarks and possible applications in Section 5.

\section{Matrix representation of a Gromov product structure}
\def\theequation{2.\arabic{equation}}\makeatother
\setcounter{equation}{0}

We recall that the decomposition of finite metric spaces is defined via the
decomposition of a polyhedral cone called the ``metric fan" \cite{Sturm}.
This classification is defined as follows.  Let $(X_n,d)$ be a finite metric space and let $K_n$ be the complete graph with $n$ vertices $V_i$.
A subgraph $G$ of $K_n$ is said to be ``a cell for the metric $d$",  if there are (positive) numbers $\{x_1,\dots,x_n\}$, such that
$d_{ij}<x_i+x_j$ and $d_{ij}=x_i+x_j$ if and only if $V_i$ and $V_j$ belong to the subgraph $G$.  Two metric spaces $(X_n,d)$ and $(X_n,d')$
are equivalent if the metrics $d$ and $d'$ give the same cell decomposition.
In preliminary work it has been shown that the Gromov product structure determines part of the cell decomposition, 
hence two metrics that belong the same class in the sense of giving the same cell decomposition, have the same Gromov product structure.



In this work we define  an $n\times n$ matrix   that carries the information of the Gromov product structure  and we obtain certain invariants of the Gromov product structures, hence of the
metric classes using this matrix representation.
We recall that two matrices $A$ and $B$ are called isospectral, if their spectrum, i.e, the set of eigenvalues are the same.
Two matrics are called similar, if there is an invertible matrix $P$ such that $B=PAP^{-1}$. If $A$
and $B$ are matrix representations of equivalent Gromov product structures, then the matrix $P$ is a permutation matrix.
Thus, the matrices of inequivalent Gromov product structures may be isospectral, even similar .

Given an undirected graph $G$, with vertices $P_i$, $i=1,\dots n$, $N$ edges and $E_{ij}$ joining $P_i$ to $P_j$, there are two basic matrix representations. One is the  $n\times n$, ``adjacency  matrix"
$A_G$, defined by $A_{ij}=1$ if there is an edge connecting $P_i$ to $P_j$, and zero otherwise.  The second is the ``incidence matrix",   an $n\times N$  matrix $B_G$, whose columns are labeled by the edges and $B_{ij}=1$ if the vertex $P_i$ is connected to the edge $j$ and zero otherwise. The adjacency matrix of a simple graph is a symmetric $n\times n$ matrix. Hence, it has real eigenvalues and a complete set of eigenvectors.  Two graphs $G$ and $G'$ are called ``isospectral" if they have the same set of eigenvalues and they are called ``equivalent"  if they can be obtained from  each other by a permutation of indices. Thus, isospectral graphs are equivalent, if the change of basis matrix of their adjacency matrices is a permutation matrix.
We will use the same characterization for the matrix representation of Gromov product structures, defined below.

\vskip 0.2cm
\noindent
{\bf Definition.}
Let
$S=\{\Delta_{1a_1b_1},\Delta_{2a_2b_2},\dots,\Delta_{na_nb_n}\}$
be a Gromov product structure for a finite metric space with $n$ elements.
The matrix representation for $S$ is the matrix $G_S$, defined by
$G_S(i,j)=1$, if the $j=a_i$ or $j=b_i$, in $\Delta_{ia_ib_i}$, and
$G_S(i,j)=0$ otherwise.
\vskip 0.2cm
Thus, $G_S$ consists of zeros and ones only and it
has exactly two $1$'s in each row.

\vskip 0.2 cm
\noindent {\bf Example: $4$-point spaces. }
For $n=4$,  the Gromov product structure (unique up to permutation of indices) is
$$S=\{\Delta_{124},\Delta_{213},\Delta_{324},\Delta_{413}\}.$$
Thus, the edges $P_1P_3$ and $P_2P_4$ are removed to lead to the pendant-free reduction.
In this case, it can be seen that the adjacency matrix of the pendant-free reduction $A_S$ and the matrix of the Gromov product structure coincide.
$$A_S=G_S=\left(
\begin{array}{cccc}
0 & 1 & 0 & 1 \\
1 & 0 & 1 & 0 \\
0 & 1 & 0 & 1 \\
1 & 0 & 1 & 0 \\
\end{array}
\right)
$$
It is easy to see that the  characteristic polynomial $k(t)$ and the minimal
polynomial $m(t)$ are
$$k(t)=(t-2)(t^2)t^2,\quad\quad m(t)=(t-2)(t^2)t.$$

\vskip 0.2 cm
\noindent {\bf Example: $5$-point spaces. }
For $n=5$, there are $3$ Gromov product structures that correspond to the
metric classes.  These are
\begin{eqnarray*}
	S_1=\{\Delta_{125},\Delta_{213},\Delta_{324},\Delta_{435},\Delta_{514}\},\quad
	S_2=\{\Delta_{125},\Delta_{213},\Delta_{325},\Delta_{425},\Delta_{514}\},\quad
	S_3=\{\Delta_{124},\Delta_{213},\Delta_{324},\Delta_{413},\Delta_{513}\}.
\end{eqnarray*}
The adjacency matrices for the pendant free reductions are given below.
$$A_{S_1}=\left(
\begin{array}{ccccc}
0 & 1 & 0 & 0 & 1\\
1 & 0 & 1 & 0 & 0\\
0 & 1 & 0 & 1 & 0\\
0 & 0 & 1 & 0 & 1\\
1 & 0 & 0 & 1 & 0\\
\end{array}        \right),\quad
A_{S_2}=   \left(
\begin{array}{ccccc}
0 & 1 & 0 & 0 & 1\\
1 & 0 & 1 & 1 & 0\\
0 & 1 & 1 & 0 & 1\\
0 & 1 & 1 & 0 & 1\\
1 & 0 & 1 & 1 & 0\\
\end{array}         \right),\quad
A_{S_3}=   \left(
\begin{array}{ccccc}
0 & 1 & 0 & 1 & 1\\
1 & 0 & 1 & 0 & 1\\
0 & 1 & 0 & 1 & 1\\
1 & 0 & 1 & 0 & 1\\
1 & 1 & 1 & 1 & 0\\
\end{array}         \right).
$$
The matrix of the Gromov product structure for $S_1$ coincides with the adjacency matrix of the pendant free reduction, but this is not the case for $S_2$ and $S_3$, as seen below.
$$G_{S_1}=\left(
\begin{array}{ccccc}
0 & 1 & 0 & 0 & 1\\
1 & 0 & 1 & 0 & 0\\
0 & 1 & 0 & 1 & 0\\
0 & 0 & 1 & 0 & 1\\
1 & 0 & 0 & 1 & 0\\
\end{array}        \right),\quad
G_{S_2}=   \left(
\begin{array}{ccccc}
0 & 1 & 0 & 0 & 1\\
1 & 0 & 1 & 0 & 0\\
0 & 1 & 0 & 0 & 1\\
0 & 1 & 0 & 0 & 1\\
1 & 0 & 0 & 1 & 0\\
\end{array}         \right),\quad
G_{S_3}=   \left(
\begin{array}{ccccc}
0 & 1 & 0 & 1 & 0\\
1 & 0 & 1 & 0 & 0\\
0 & 1 & 0 & 1 & 0\\
1 & 0 & 1 & 0 & 0\\
1 & 0 & 1 & 0 & 0\\
\end{array}         \right).$$
Their characteristic and minimal polynomials are given below:
\begin{eqnarray*}
	k_1(t)=(t-2)(t^2+t-1)^2,\quad	k_2(t)=(t-2)(t+2)t^3,\quad 	k_3(t)=(t-2)(t+2)t^3,\\
    m_1(t)=(t-2)(t^2+t-1),\quad      m_2(t)=(t-2)(t+2)t^2,\quad  m_3(t)=(t-2)(t+2)t.
\end{eqnarray*}
\vskip0.2cm

We now prove that if two Gromov product structures are equivalent, then their matrices are similar and the change of basis matrix is the corresponding permutation matrix.

\begin{proposition}
	Let $(X,d)$ and $(X',d')$ be two metric spaces with Gromov product structures
	$S=\{\Delta_{1i_1j_1},\dots ,\Delta_{ni_nj_n}\}$ and $S'=\{\Delta_{1a_1b_1},\dots ,\Delta_{na_nb_n}\}$. Let $G_S$  and $G_{S'}$ be the  matrices defined by
	$G_S(i,j)=1$, if the $j=a_i$ or $j=b_i$, in $\Delta_{ia_ib_i}$, and
	$G_S(i,j)=0$ otherwise.
	Assume that  there is a permutation matrix $P$ that such $S$ is mapped to $S'$. Then  $G_S P=P G_{S'}$.
\end{proposition}

\begin{proof}
	Assume that $\Delta_{iab}$ is minimal at $P_i$, and the $\{i,a,b\}$ is mapped to $\{s,p,q\}$ under the permutation $P$. Then
	$P_{i,s}=P_{a,p}=P_{b,q}=1$ and these are the only nonzero entries in the rows $i$, $a$ and $b$ respectively.
	In component form, the matrix equation $G_S P=P G_{S'}$ is written as
	$$\sum_k (G_S)_{ik} P_{kj}=\sum_l P_{il} (G_{S'})_{lj}.$$
	Since  $\Delta_{i,a,b}$ is minimal at $P_i$, the first sum contains only $2$ terms. Furthermore, since $i$ is mapped to $s$ the second sum has a single nonzero term. Thus we have
	$$(G_S)_{ia} P_{aj}+(G_S)_{ib} P_{bj}= P_{is} (G_{S'})_{sj}.$$
	The left hand side is nonzero only for $j=a$ or $j=b$.  In either case, the left hand side is also nonzero, since $(G_{S'})_{sp}=(G_{S'})_{sq}=1$, and both sides are zero otherwise.  It follows that $G_S P=P G_{S'}$.
\end{proof}

\section {Matrix invariants}

In this section we  relate the invariants of the matrix representation to certain properties of the Gromov product structures. The proofs are straightforward and they are omitted. 

\begin{proposition} 
The rank of $G$ is equal to the number of edges removed in the pendant-free reduction.
\end{proposition}

The structure of the removed set as a graph also carries information, hence adjacency matrices of the removed set can be considered for constructing invariants, but these are not considered here. 
\vskip 0.3cm
Recall that two matrices are called isospectral, if they have the same eigenvalues.  The characteristic polynomial of an $n\times n$ matrix $A$ is defined as
$$K_A(t)=det(tI-A)=\sum_{k=1}^n c_k t^k,$$
where  $I$ is the  identity matrix.  It is well known that the coefficients $c_k$'s are related to the traces of the powers of $A$,
through the method of Leverrier or the method of Faddeev,  \cite{Gantmacher}. Although it would be possible to compute the eigenvalues, it is preferable to work
with the traces of the powers of $G_S$, that are integers. 

\begin{proposition} 
	Gromov products structures with the same $K$
	$$K=\{trace(G),\quad  trace(G^2),\dots , trace(G^n)\}$$
	are isospectral.
\end{proposition}

Isospectral matrices can further be decomposed into finer equivalence classes defined  in terms of their minimal polynomials.
For $n=5$ and $n=6$ we have given the characteristic and minimal polynomials for each class of Gromov product structure, but in general, the
determination of the minimal polynomial is not practical and it is not attempted as a classification tool for $n=7$ and $n=8 $
\vskip 0.3cm
In order to define the last two invariants, we need to define ``chains" and ``cycles" of Gromov products.
If the set of minimal Gromov products at $P_{i_1},P_{i_2},\dots P_{i_k}$ are
$$\{\Delta_{i_1ai_2},\Delta_{i_2i_1i_3},\dots \Delta_{i_ki_{k-1}b}\}$$
then they are said to form a ``chain" of length $k$.  If $a=b$ they are said to form a ``cycle" of length $k$.
The set of minimal Gromov products is decomposed to a disjoint union of chains and cycles.
A chain is said to have $2$ ``ends", while a cycle has no ``ends".
\vskip 0.3cm 
The  Gromov product structure $S$ of an  $n$-point space is said to be ``reducible", if there is a $k$-element  subset of the index set such that
the restriction of $S$ to this subset is a Gromov product structure on a $k$-point space.
The reducibility and irreducibility of a Gromov product structure is defined in terms of its matrix as follows.

\begin{proposition}\label{reducible} If $S$ is an irreducible Gromov product structure, then the matrix $P_S$ defined by 
	$$P_S=I+G+G^2+\dots G^n$$
is irreducible. 	
\end{proposition}

The number of ``ends" is an invariant of the Gromov product structure.  It is related to the matrix representation as follws. 

\begin{proposition}\label{reducible} Let
	$B=(G+G^t)$. The symmetric and antisymmetric parts of $G$ are given respectively by  $C=floor(B/2)$ and its   $D=G-C$.
	Then, the total number of $1$'s in $D$ is equal to  the number of ``ends"  of the Gromov product chains.
\end{proposition}

\section{Equivalence Class Decomposition of $6$-Point  Metric Spaces via Gromov Product Structures}

The classification of $6$-point metric spaces with respect to the metric fan decomposition and a list of the corresponding types
are given in \cite{Sturm} where it is shown that there are $339$
combinatorial types. By a straightforward
application of Proposition 1 and Corollary 1 of \cite{BCK2017}, we have obtained $32$ allowable Gromov product structures.  Among these we eliminated the ones that are not generic and
by comparing these with the list given in \cite{Sturm}, we've identified the
decomposition of the combinatorial types into Gromov product structures.

We collect the isospectral structures and give their minimal polynomials in Tables 1 and 2, respectively for reducible and irreducible structures. All characteristic and minimal polynomials have a common factor $(t-2)$. 
The number of removed edges and the number of  ``ends" in the collection of Gromov products chains and cycles are also given in these tables

\section{Concluding remarks}
The representation of Gromov Product Structures  by a square matrix has a number of advantages. Finding the equivalence of metric structures under the permutation of the vertices is a basic and difficult problem. As there is a one-to-one correspondence between the Gromov Product Structure and its matrix representation, the equivalence problem is reduced to checking whether the corresponding matrices are similar via a permutation matrix.

We recall that isospectral matrices are characterized by the equality of their characteristic polynomials, whose coefficients are integers.  Thus, the characteristic polynomial of the representation matrix gives rise to coarser decomposition into isospectral matrices.  It is then much simpler to search for those that are similar via a permutation matrix.

In addition to the characteristic polynomial, matrix invariants such rank, reducibility or irreducibility allow to give rigorous definitions for various characteristics of metric spaces and simplifies further the solution of the equivalence problem.


The matrix representation reduce the computational complexity of the equivalence problem as follows.
For an $N$ point space, the order of the permutation group is $N!$.   If there are $M$ objects to check for equivalence,  one has to apply  $N!$  elements of the permutation group to  each pair  $M(M-1)/2$. If the number of objects to compare for equivalence is reduced to $M/n$, then the number of comparisons is reduced by $n^2$.
The solution of the equivalence problem for $7$ point spaces became possible with the help of the improvement due to spectral decomposition.

\bibliography{mybibfile}

\begin{thebibliography}{00}
	\bibitem{Sturm} Sturmfels, B., J. Yu, Classification of
	Six-point Metrics, The Electronic Journal of Combinatorics, 11, (2004), R44.
	\bibitem{BCK2017} Bilge, A. H., D. Celik and S. Kocak, An Equivalence Class Decomposition of Finite Metric Spaces via Gromov Products, Discrete Mathematics, Vol. 340, (2017), 1928-1932.
	\bibitem{IB_algorithm} Incegul, M., Algorithms for the Equivalence Class Decomposition of Finite Metric Spaces via Gromov Products, Unpublished manuscript (2019).
	\bibitem{Koolen} Koolen, J., A. Lesser, V. Moulton, Optimal
	realizations of generic five-point metrics, European J. of
	Combin., 30,(2009) 1164-1171.


\bibitem{Gantmacher} Gantmacher, F.R., The Theory of Matrices, Chelsea Publishing Company,Vol. 1, (1977).


%
\bibitem{VJ} Vaisala, J., ``Gromov hyperbolic spaces" http://www.helsinki.fi/~jvaisala/grobok.pdf, (2004)
%
\end{thebibliography}


\begin{table}[ht]
	\caption{Minimal Gromov products at node $P_i$: Reducible structures: Characteristic and  minimal polynomials $K(t)$ and $m(t)$, the number of removed edges $N_r$, the number of ``ends" $N_e$ } 
	\centering 
	\begin{tabular}{|c| c  |l|l|c|c| }
		\hline
		{\small }&{\small Gromov Product Structure} & {\small -k(t)/(t-2)}&{\small -m(t)/(t-2)}& {\small $N_r$} & {\small $N_e$} \\ \hline
		$R_1$&$\Delta_{124},\Delta_{213},\Delta_{324},\Delta_{413},\Delta_{513},\Delta_{613}$&$(t+2)t^4$&$(t+2)t^4$ &2&  4  		             \\\hline
		$R_2$&$\Delta_{124},\Delta_{213},\Delta_{324},\Delta_{435},\Delta_{524},\Delta_{624}$&$(t+2)t^4$&$(t+2)t^4$	&3&	 4              \\\hline
		$R_3$&$\Delta_{124},\Delta_{213},\Delta_{324},\Delta_{413},\Delta_{513},\Delta_{624}$&$(t+2)t^4$&$(t+2)t^4$	&2&  4                 \\\hline
		$R_4$&$\Delta_{124},\Delta_{213},\Delta_{324},\Delta_{435},\Delta_{524},\Delta_{613}$&$(t+2)t^4$&$(t+2)t^4$	&3&	 4                 \\\hline
		$R_5$&$\Delta_{124},\Delta_{213},\Delta_{324},\Delta_{413},\Delta_{513},\Delta_{625}$&$(t+2)t^4$&$(t+2)t^4$	&3&  4                 	\\\hline
		$R_6$&$\Delta_{124},\Delta_{213},\Delta_{324},\Delta_{435},\Delta_{524},\Delta_{615}$&$(t+2)t^4$&$(t+2)t^4$	&4&  4                	\\\hline
		$R_7$&$\Delta_{124},\Delta_{213},\Delta_{324},\Delta_{413},\Delta_{516},\Delta_{625}$&$(t+2)(t-1)(t+1)t^2$&$(t+2)(t-1)(t+1)t^2$ &4&2      		\\\hline
		$R_8$&$\Delta_{125},\Delta_{213},\Delta_{324},\Delta_{435},\Delta_{514},\Delta_{613}$&$(t^2+t-1)^2t   	$&$(t^2+t-1)^2t$	     &4&2       \\\hline
		$R_9$&$\Delta_{124},\Delta_{213},\Delta_{324},\Delta_{413},\Delta_{516},\Delta_{635}$&$(t+2)(t-1)(t+1)t^2$&$(t+2)(t-1)(t+1)t^2$	  &4&2        	\\\hline
	\end{tabular}
\end{table}

\begin{table}[h]
	\caption{Minimal Gromov products at node $P_i$: Irreducible structures: Characteristic and  minimal polynomials $K(t)$ and $m(t)$, the number of removed edges $N_r$, the number of ``ends" $N_e$  } 
	\centering 
	\begin{tabular}{|c| c  |l|l|c|c| }
		\hline
		{\small } &{\small Gromov Product Structure} &		{\small -k(t)/(t-2)}&{\small -m(t)/(t-2)}& {\small $N_r$} & {\small $N_e$}\\ \hline
		$I_1$&$   \Delta_{124},\Delta_{213},\Delta_{324},\Delta_{456},\Delta_{524},\Delta_{624}$&$(t+2)t^4    		  $&$(t+2)t^4    		 $  &3&4  \\\hline	
		$I_2$&$   \Delta_{124},\Delta_{213},\Delta_{324},\Delta_{456},\Delta_{513},\Delta_{624}$&$(t+1)^2 t^3        $&$(t+1)^2*t^3          $  &3&6       \\\hline 
		$I_3$&$   \Delta_{124},\Delta_{213},\Delta_{324},\Delta_{435},\Delta_{516},\Delta_{624}$&$(t+1)^2 t^3		  $&$(t+1)^2*t^3		 $  &4&6      \\\hline 
		$I_4$&$   \Delta_{124},\Delta_{213},\Delta_{324},\Delta_{456},\Delta_{513},\Delta_{613}$&$(t^2+2 t+2)t^3     $&$(t^2+2*t+2)t^3       $  &3&8      \\\hline 
		$I_5$&$   \Delta_{124},\Delta_{213},\Delta_{324},\Delta_{435},\Delta_{526},\Delta_{635}$&$(t+2)t^4  		  $&$(t+2)t^4  		     $  &4&4        \\\hline 
		$I_6$&$   \Delta_{125},\Delta_{213},\Delta_{324},\Delta_{436},\Delta_{536},\Delta_{625}$&$(t+2)t^4  		  $&$(t+2)t^4  		     $  &4&4        \\\hline 
		$I_7$&$   \Delta_{124},\Delta_{213},\Delta_{324},\Delta_{435},\Delta_{546},\Delta_{635}$&$(t+2)(t-1)(t+1)t^2 $&$(t+2)(t-1)(t+1)t^2   $  &4&2          \\\hline 
		$I_8$&$   \Delta_{124},\Delta_{213},\Delta_{324},\Delta_{435},\Delta_{546},\Delta_{615}$&$(t+2)(t-1)(t+1)t^2 $&$(t+2)(t-1)(t+1)t^2   $  &5&2      \\\hline 
		$I_9$&$   \Delta_{124},\Delta_{213},\Delta_{324},\Delta_{435},\Delta_{516},\Delta_{635}$&$(t+1)(t^2+t-1)t^2  $&$(t+1)(t^2+t-1)t^2    $&4&4\\\hline 
		$I_{10}$&$\Delta_{124},\Delta_{213},\Delta_{324},\Delta_{435},\Delta_{516},\Delta_{625}$&$(t^4+2t^3-t+1)t	  $&$(t^4+2t^3-t+1)t	 $&5&4\\\hline 
		$I_{11}$&$\Delta_{125},\Delta_{213},\Delta_{324},\Delta_{435},\Delta_{546},\Delta_{613}$&$(t+1)(t^2+t-1)t^2  $&$(t+1)(t^2+t-1)t^2    $&5&4\\\hline 
		$I_{12}$&$\Delta_{124},\Delta_{213},\Delta_{324},\Delta_{435},\Delta_{526},\Delta_{615}$&$(t+2)t^4		      $&$(t+2)t^4		     $&5&4\\\hline 
		$I_{13}$&$\Delta_{124},\Delta_{213},\Delta_{324},\Delta_{435},\Delta_{546},\Delta_{613}$&$(t+2)t^4		      $&$(t+2)t^4		     $&4&4\\\hline 
		$I_{14}$&$\Delta_{124},\Delta_{213},\Delta_{324},\Delta_{435},\Delta_{546},\Delta_{625}$&$(t^2+t-1)^2t	      $&$(t^2+t-1)^2t	     $&5&2\\\hline 
		$I_{15}$&$\Delta_{125},\Delta_{213},\Delta_{324},\Delta_{435},\Delta_{546},\Delta_{625}$&$(t^2+t-1)^2t	      $&$(t^2+t-1)^2t	     $&5&2\\\hline
		$I_{16}$&$\Delta_{156},\Delta_{213},\Delta_{324},\Delta_{456},\Delta_{513},\Delta_{624}$&$(t+1)^2t^3	      $&$(t+1)^2t^3	         $&3&6\\\hline
		$I_{17}$&$\Delta_{126},\Delta_{213},\Delta_{324},\Delta_{435},\Delta_{546},\Delta_{615}$&$(t+2)(t-1)^2(t+1)^2$&$(t+2)(t-1)^2(t+1)^2  $&6&0\\\hline
	\end{tabular}
	\label{table:nonlin} 
\end{table}

\end{document}